\documentclass[12pt]{article}
\usepackage{amssymb,amsmath}
\usepackage{latexsym,bm}

\usepackage{amsfonts}
\usepackage{latexsym,amsmath,amssymb,amsfonts,epsfig,psfrag,url,graphics,ifpdf,multicol}
\usepackage{eepic,color,colordvi,amscd,amsthm}
\usepackage[section]{algorithm}
\usepackage{algpseudocode}
\usepackage[numbers,sort&compress]{natbib}
\usepackage{indentfirst,graphics,epsfig}
\usepackage{graphicx}
\usepackage{graphics}
\usepackage{tikz}
\usepackage{subfigure}

\newtheorem{theo}{Theorem}
\newtheorem{lem}[theo]{Lemma}

 \theoremstyle{definition}

\theoremstyle{remark}

\newcounter{casenum}[theo]

\newcounter{subcasenum}[theo]

\newcounter{claimnum}[theo]

\allowdisplaybreaks
\usepackage{indentfirst,subfig}

\begin{document}
\thispagestyle{plain}

\begin{center} {\Large Isolation of the diamond graph
}
\end{center}
\pagestyle{plain}
\begin{center}
{
  {\small  Jingru Yan \footnote{ E-mail address: mathyjr@163.com}}\\[3mm]
  {\small  Department of Mathematics, East China Normal University, Shanghai 200241,  China }\\

}
\end{center}

\begin{center}

\begin{minipage}{140mm}
\begin{center}
{\bf Abstract}
\end{center}
{\small   A graph is $H$-free if it does not contain $H$ as a subgraph. The diamond graph is the graph obtained from $K_4$ by deleting one edge. We prove that if $G$ is a connected graph with order $n\geq 10$, then there exists a subset $S\subseteq V(G)$ with $|S|\leq n/5$ such that the graph induced by $V(G)\setminus N[S]$ is diamond-free, where $N[S]$ is the closed neighborhood of $S$. Furthermore, the bound is sharp.

 }

\end{minipage}
\end{center}

\section{Introduction}
Let $G$ be a finite simple graph with vertex set $V(G)$ and edge set $E(G)$. The order and the size of a graph $G$, denoted $|V(G)|$ and $|E(G)|$, are its number of vertices and edges, respectively. For a subset $S\subseteq V(G)$, the neighborhood of $S$ is the set $N_G(S)=\{u\in V(G)\setminus S\mid uv\in E(G), v\in S\}$ and closed neighborhood of $S$ is the set $N_G[S]=N_G(S)\cup S$. Thus $N_G(v)$ and $N_G[v]$ denote the neighborhood and closed neighborhood of $v\in V(G)$, respectively. The degree of $v$ is $d_G(v)=|N_G(v)|$. If the graph $G$ is clear from the context, we will omit it as the subscript. $\delta(G)$ and $\Delta(G)$ denote the minimum and maximum degree of a graph $G$, respectively. Denote by $G[S]$ the subgraph of $G$ induced by $S \subseteq V(G)$. For terminology and notations not explicitly described in this paper, readers can refer to related books \cite{BM,W}.

Given graphs $G$ and $H$, the notation $G + H$ means the \emph{disjoint union} of $G$ and $H$. Then $tG$ denotes the disjoint union of $t$ copies of $G$. For graphs we will use equality up to isomorphism, so $G = H$ means that $G$ and $H$ are isomorphic. A graph is $H$-free if it does not contain $H$ as a subgraph.  $\kappa(G)$ and $\gamma(G)$ denote the connectivity and domination number of a graph $G$, respectively. $P_n, C_n, K_n$ and $K_{p,q}$ stand for the \emph{path}, \emph{cycle}, \emph{complete graph} of order $n$ and \emph{complete bipartite graph} with partition sets of $p$ and $q$ vertices, respectively.

Let $G$ be a graph and $\mathcal{F}$ a family of graphs. A subset $S\subseteq V(G)$ is called an $\mathcal{F}$-isolating set of $G$ if $G-N[S]$ contains no subgraph isomorphic to any $F \in\mathcal{F}$. The minimum cardinality of an $\mathcal{F}$-isolating set of a graph $G$ will be denoted $\iota(G,\mathcal{F})$ and called the $\mathcal{F}$-isolation number of $G$. When $\mathcal{F}=\{H\}$, we simply write $\iota(G,H)$ for $\iota(G,\{H\})$.

The definition of isolation set is a natural extension of the commonly defined dominating set, which was introduced by Caro and Hansberg \cite{CH}. Indeed, if $\mathcal{F}= \{K_1\}$, then an $\mathcal{F}$-isolating set coincides with a dominating set and $\iota(G,\mathcal{F})=\gamma(G)$. A classical result of Ore \cite{O} is that the domination number of a graph $G$ with order $n$ and $\delta(G)\geq 1$ is at most $n/2$. In other words, if $G$ is a connected graph of order $n\geq 2$, then $\iota(G,K_1)\leq n/2$. Caro and Hansberg \cite{CH} focused mainly on $\iota(G,K_2)$ and $\iota(G,K_{1,k+1})$ and gave some basic properties, examples concerning $\iota(G,\mathcal{F})$ and the relation between $\mathcal{F}$-isolating sets and dominating sets. They \cite{CH} proved that if $G$ is a connected graph of order $n\geq 3$ that is not a $C_5$, then $\iota(G,K_2)\leq \frac{n}{3}$. Since then, Borg \cite{B} showed that if $G$ is a connected graph of order $n$, then $\iota(G,\{C_k : k\geq 3\})\leq \frac{n}{4}$ unless $G$ is $K_3$. After that, Borg, Fenech and Kaemawichanurat \cite{BFK} proved that if $G$ is a connected graph of order $n$, then $\iota(G,K_k)\leq \frac{n}{k+1}$ unless $G$ is $K_k$, or $k = 2$ and $G$ is $C_5$. Both the bounds are sharp. Then Zhang and Wu \cite{ZW} gave the result that if $G$ is a connected graph of order $n$, then $\iota(G,P_3)\leq \frac{2n}{7}$ unless $G\in \{P_3,C_3,C_6\}$, and this bound can be improved to $\frac{n}{4}$ if $G\notin \{P_3,C_7,C_{11}\}$ and the girth of $G$ at least 7. For more research on isolation set, refer to \cite{STP,BK,FK}.

The \emph{diamond graph} is the graph obtained from $K_4$ by deleting one edge (see Figure 1). The \emph{book graph} with $p$ pages, denoted $B_p$, is the graph that consists of $p$ triangles sharing a common edge. Obviously, $B_2$ is the diamond graph. For the convenience of expression, we use $B_2$ to represent the diamond graph in the sequel.

\begin{figure}[tp]
\subfigure
  {
  \begin{minipage}[b]{.4\linewidth}
    \centering
    \includegraphics[scale=0.7]{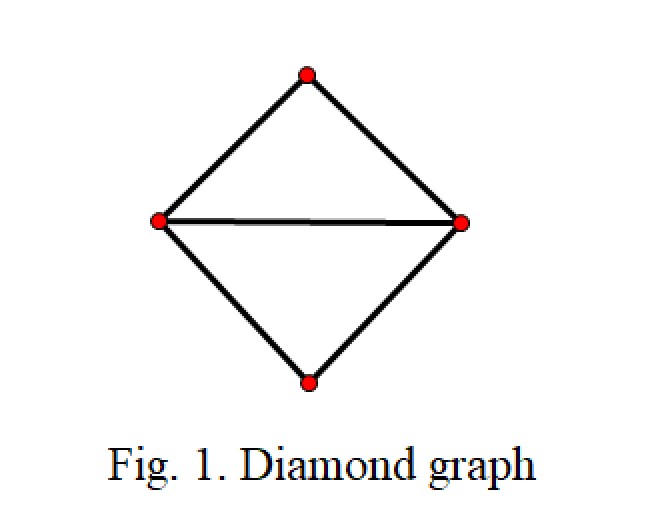}
  \end{minipage}
  }
\subfigure
  {
  \begin{minipage}[b]{.6\linewidth}
    \centering
    \includegraphics[scale=0.7]{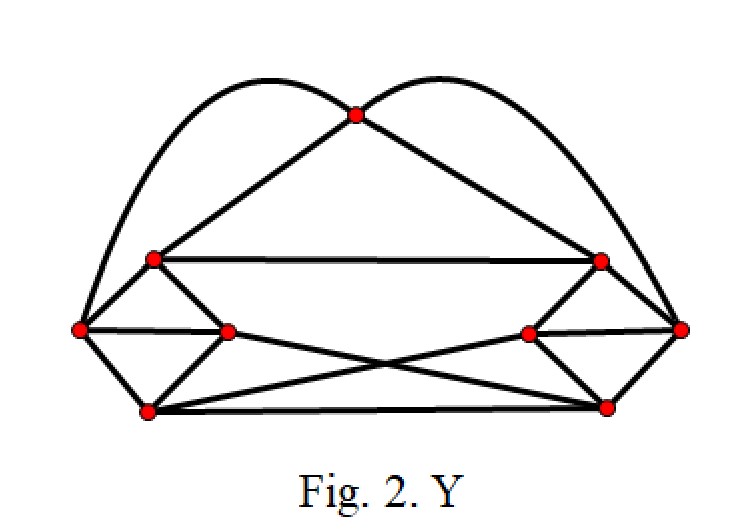}
  \end{minipage}
  }
\end{figure}

In this paper, we consider the isolation number of the diamond graph in a connected graph of a given order.
\begin{theo}\label{th1}
If $G$ is a connected graph of order $n$, then, unless $G$ is the diamond graph, $K_4$, or $Y$,
$$\iota(G,B_2)\leq \frac{n}{5},$$ where $Y$ is shown in the Figure 2.
\end{theo}

\section{Main results}
From the proof of Theorem 3.8 in this paper \cite{CH}, we obtain Lemma \ref{lem9} and give an example that satisfies the Lemma, see Figure 3. The minimum cardinality of a $B_2$-isolating set of the graph $H$ of order 15 is 3.

\begin{lem}\label{lem9}
There exists a connected graph $G$ of order $n$ such that $\iota(G,B_2)= \frac{n}{5}$.
\end{lem}

\begin{figure}[h]
  \centering
  \includegraphics[scale=0.8]{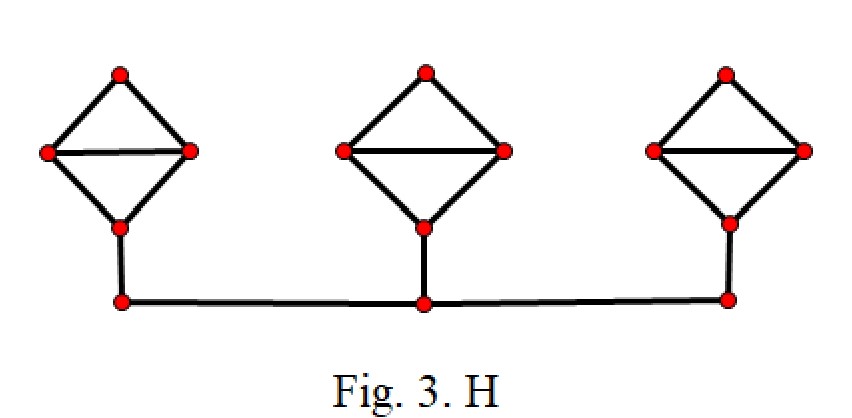}

\end{figure}

We start with two lemmas that will be used repeatedly.

\begin{lem}\label{lem1}\cite{B}
If $G$ is a graph, $\mathcal{F}$ is a set of graphs, $A\subseteq V(G)$, and $B \subseteq N[A]$, then
$$\iota(G,\mathcal{F})\leq |A|+ \iota(G-B,\mathcal{F}).$$ In particular, if $A=\{v\}$ and $B=N[A]$, then $\iota(G,\mathcal{F})\leq 1+ \iota(G-N[v],\mathcal{F})$.
\end{lem}

\begin{lem}\label{lem2}\cite{B}
If $G_1,\ldots,G_k$ are the distinct components of a graph $G$, then  $$\iota(G,\mathcal{F})=\sum_{i=1}^{k}\iota(G_i,\mathcal{F}).$$
\end{lem}

For any graph $G$, let $A,B\subseteq V(G)$ and $A\cap B=\phi$. Denote by $E(A,B)$ the set of edges of $G$ with one end in
$A$ and the other end in $B$ and $e(A,B)=|E(A,B)|$. We abbreviate $E(\{x\},B)$ to $E(x,B)$ and $e(\{x\},B)$ to $e(x,B)$.

Now, we first prove Theorem \ref{th1} when the order $n\leq 9$.
\begin{lem}\label{lem3}
Let $G$ be a connected graph of order $n\leq 9$. Then
$\iota(G,B_2)\leq \frac{n}{5}$ except for $G\in \{B_2,K_4,Y\}$.
\end{lem}
\begin{proof}
Let $G$ be a connected graph of order $n$. The result is trivial if $n \leq 4$ or $\iota(G, B_2) = 0$. Suppose $5\leq n\leq 9$ and $\iota(G, B_2) \geq 1$. Then we need to show that $G$ has a $B_2$-isolating set $S$ with $|S|=1$ except for $G=Y$.

Since $\iota(G, B_2) \geq 1$, it follows $G$ contains $B_2$ and $\Delta(G)\geq 3$. Let $x\in V(G)$ such that $d(x)=\Delta(G)$. Of course, $S=\{x\}$ is a $B_2$-isolating set of $G$ if $G-N[x]$ is $B_2$-free. Otherwise, it implies that $n=8$ with $\Delta(G)= 3$ or $n=9$ with $3\leq \Delta(G)\leq 4$. We distinguish two cases.

{\it Case 1.} $\Delta(G)= 3$ and $n=8$ or $9$

Let $u\in V(G)$ such that $d(u)=3$ and $G[N[u]]=B_2$. If $G-N[u]$ is $B_2$-free, then $\iota(G, B_2) = 1\leq \frac{n}{5}$. So, suppose that $G-N[u]$ contains $B_2$. For $n=8$, obviously, $G-N[u]=B_2$. Since $G$ is a connected graph and $\Delta(G)= 3$, there is an edge $e=yz$ with $y\in N(u)$ and $z\in V(G)\setminus N[u]$. It is easy to check that $\{y\}$ or $\{z\}$ is a $B_2$-isolating set of $G$. Hence $\iota(G, B_2)=1 \leq \frac{n}{5}$. Now we prove the case of $n=9$. Let us consider a copy $H$ of $B_2$ in $G-N[u]$ and let $w$ be the remaining vertex of $G-N[u]-V(H)$. If there is an edge $e=yz$ with $y\in N(u)$ and $z\in V(H)$, then $\{y\}$ or $\{z\}$ is a $B_2$-isolating set of $G$. Otherwise, $w$ is a cut-vertex of $G$ and $G-N[w]$ is $B_2$-free. Hence $\iota(G, B_2)=1 \leq \frac{n}{5}$.

{\it Case 2.} $\Delta(G)= 4$ and $n=9$

Let $u\in V(G)$ such that $d(u)=4$ and let $F=G-N[u]$. We have $\iota(G, B_2)=1$ if $F$ is $B_2$-free. Assume that $F$ contains $B_2$. Since $|V(F)|=4$, then $F=B_2$ or $F=K_4$. The vertices are labeled as shown in the Figure 4. We distinguish two subcases.

\begin{figure}[h]
\subfigure
  {
  \begin{minipage}[b]{.4\linewidth}
    \centering
    \includegraphics[scale=0.8]{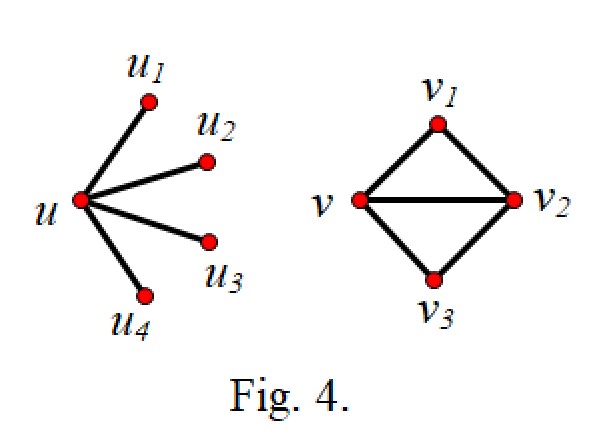}
  \end{minipage}
  }
\subfigure
  {
  \begin{minipage}[b]{.6\linewidth}
    \centering
    \includegraphics[scale=0.8]{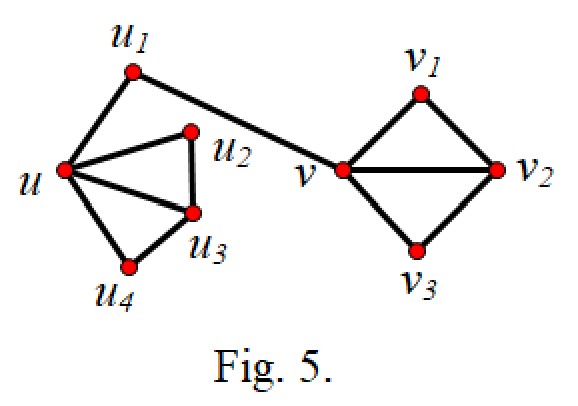}
  \end{minipage}
  }
\end{figure}

{\it Subcase 2.1.} $F=K_4$. Note that $e(N(u),V(F))\neq 0$. Without loss of generality, suppose $u_1$ is adjacent to $v$. If $G-N[v]$ is $B_2$-free, $\iota(G, B_2)=1$. Otherwise, $G-N[v]=K_4$ or $B_2$ since $|V(G)\setminus N[v]|=4$. For $G-N[v]=K_4$, we have $G-\{u,u_1,v\}$ is $B_2$-free. Thus, $\{u_1\}$ is a $B_2$-isolating set of $G$ and $\iota(G, B_2) \leq \frac{n}{5}$. For $G-N[v]=B_2$, $G-\{u,u_1,v\}$ contains $B_2$ if and only if $u_2$ or $u_4$ is adjacent to at least two vertices of $\{v_1,v_2,v_3\}$. Now we have $\{u_2\}$ or $\{u_4\}$ is a $B_2$-isolating set of $G$ and hence $\iota(G, B_2) \leq \frac{n}{5}$.

{\it Subcase 2.2.} $F=B_2$. The proof for this case is similar to Subcase 2.1. First suppose $F$ has a vertex of degree 3 that is adjacent to one vertex of $N(u)$. Without loss of generality, suppose $v$, $d_{F}(v)=3$, is adjacent to $u_1$. Then we have $\iota(G, B_2)=1$ if $G-N[v]$ is $B_2$-free. Otherwise, $G-N[v]$ contains $B_2$. For $G-N[v]=K_4$, by the proof of Subcase 2.1, $\iota(G, B_2) \leq \frac{n}{5}$. For $G-N[v]=B_2$, $G-\{u,u_1,v\}$ contains $B_2$ when one of the following four cases is true. (1) $u_2$ is adjacent to $v_1$ and $v_2$ and $u_3$ is adjacent to $v_1$. (2) $u_2$ is adjacent to $v_2$ and $v_3$ and $u_3$ is adjacent to $v_3$. (3) $u_3$ is adjacent to $v_1$ and $u_4$ is adjacent to $v_1$ and $v_2$. (4) $u_3$ is adjacent to $v_3$ and $u_4$ is adjacent to $v_2$ and $v_3$. We can see from Figure 5 that the proof methods are similar for the four cases. So let us just consider the first case. Note that $G-N[u_2]$ contains $B_2$ if and only if $G[u_1,u_4,v_3]=K_3$. Observe that $G=Y$.

Next suppose that only the vertices of degree 2 of $F$ are adjacent to the vertices of $N(u)$. Suppose $v_1$, $d_{F}(v_1)=2$, is adjacent to $u_1$. Then $\iota(G, B_2)=1$ if $G-\{u,u_1,v_1\}$ is $B_2$-free. Otherwise, since $e(v,N(u))=e(v_2,N(u))=0$, we have $G[u_2,u_3,u_4,v_3]$ contains $B_2$ as a subgraph. Recall that $\Delta(G)= 4$, then $G[u_2,u_3,u_4,v_3]=B_2$. Moreover, $G-N[v_3]$ is $B_2$-free. Thus, $\iota(G, B_2) \leq \frac{n}{5}$.

Hence, in all cases we obtain $\iota(G, B_2) \leq \frac{n}{5}$ with $n\leq 9$ except for $\iota(B_2, B_2)=1$, $\iota(K_4, B_2)=1$ and $\iota(Y, B_2)=2$.
\end{proof}

Next, we prove Theorem \ref{th1} when $\Delta(G)= 3$.

\begin{lem}\label{lem5}
Let $G$ be a connected graph of order $n$. $\iota(G',B_2)=\iota(G,B_2)$ if\\
(1) $G'$ is obtained from $G$ by attaching one edge to any vertex of $G$,\\
(2) $G'$ is obtained from $G$ by identifying one vertex of a triangle and a vertex of $G$,\\
(3) $G'$ is obtained from $G+K_3$ by adding an edge joining a vertex of $K_3$ and a vertex of $G$.
\end{lem}

\begin{proof}
(1) Let $S$ be a minimum $B_2$-isolating set of $G$. Then, clearly, $S$ is a $B_2$-isolating set of $G'$ and thus $\iota(G',B_2)\leq \iota(G,B_2)$. Let $S'$ be a minimum $B_2$-isolating set of $G'$ and let $x$ be the vertex of $V(G')\setminus V(G)$. Note that $S'\setminus \{x\}$ is a $B_2$-isolating set of $G$. Thus, $\iota(G,B_2)\leq \iota(G',B_2)$. Now the both inequalities imply the result.

(2) and (3) can be proved similarly as (1).
\end{proof}

\begin{lem}\label{th3}
Let $G$ be a connected graph of order $n$. If $\Delta(G)= 3$, then $$\iota(G,B_2)\leq \frac{n}{5}$$ except for $G\in \{B_2,K_4\}$.
\end{lem}

\begin{proof}
Let $G$ be a connected graph of order $n$ with $\Delta(G)= 3$. The proof is by induction on $n$. By Lemma \ref{lem3}, the result is trivial if $n \leq 9$ or $\iota(G,B_2)=0$. Thus, suppose that $n\geq 10$ and $\iota(G,B_2)\geq 1$. Since $G$ contains $B_2$, it follows that there exists at least one vertex $u\in V(G)$ such that $d(u)=3$ and $G[N[u]]=B_2$. Let $N(u)=\{u_1,u_2,u_3\}$ and let $u_2$ be the another vertex of the $B_2$ with degree 3. As $G$ is connected and $\Delta(G)= 3$, then either $d(u_1)=3$ or $d(u_3)=3$. We distinguish the following two cases.

{\it Case 1.} $d(u_1)=3$ and $d(u_3)=2$ or $d(u_1)=2$ and $d(u_3)=3$

Without loss of generality, suppose $d(u_1)=3$ and $d(u_3)=2$. Let $w\in V(G-N[u])$ and $w$ is adjacent to $u_1$. Define $G'=G-N[u]-w$. Note that $|V(G')|=n-5 \geq 5$. Clearly, $\{u_1\}$ is a $B_2$-isolating set of $G$ if $G'$ is $B_2$-free. Suppose $G'$ contains $B_2$. If $G'$ is connected, by the induction hypothesis, $\iota(G',B_2)\leq \frac{n-5}{5}$. Then by Lemma \ref{lem1} and \ref{lem2}, we have $\iota(G,B_2)\leq |\{u_1\}|+\iota(G',B_2)\leq 1+\frac{n-5}{5}=\frac{n}{5}$.

Suppose that $G'$ is disconnected. It is easy to check that $d(w)=3$ and $G'$ has exactly two components. Let $G'=G_1+G_2$. If $G_1\neq B_2$ and $G_2\neq B_2$, the union of a minimum $B_2$-isolating set of $G_1$, a minimum $B_2$-isolating set of $G_2$ and $\{u_1\}$ is a $B_2$-isolating set of $G$. By the induction hypothesis and Lemma \ref{lem2},
$$\iota(G,B_2)\leq 1+\iota(G_1,B_2)+\iota(G_2,B_2)\leq 1+\frac{|V(G_1)|}{5}+\frac{|V(G_2)|}{5}=\frac{n}{5}.$$ If $G_1= B_2$ and $G_2= B_2$, we have $n=13$. Observe that $\{w\}$ is a $B_2$-isolating set of $G$. Hence $\iota(G,B_2)\leq \frac{n}{5}$. So, it remains to consider the case of exactly one of $\{G_1,G_2\}$ is isomorphic to $B_2$. Suppose that $G_1=B_2$ and $G_2\neq B_2$. Let $w_1$ be the neighbor of $w$ in $G_2$ and let $G''=G'-V(G_1)-w_1$. Note that $|V(G'')|=n-10$.

\begin{figure}[h]
  \centering
  \includegraphics[scale=0.8]{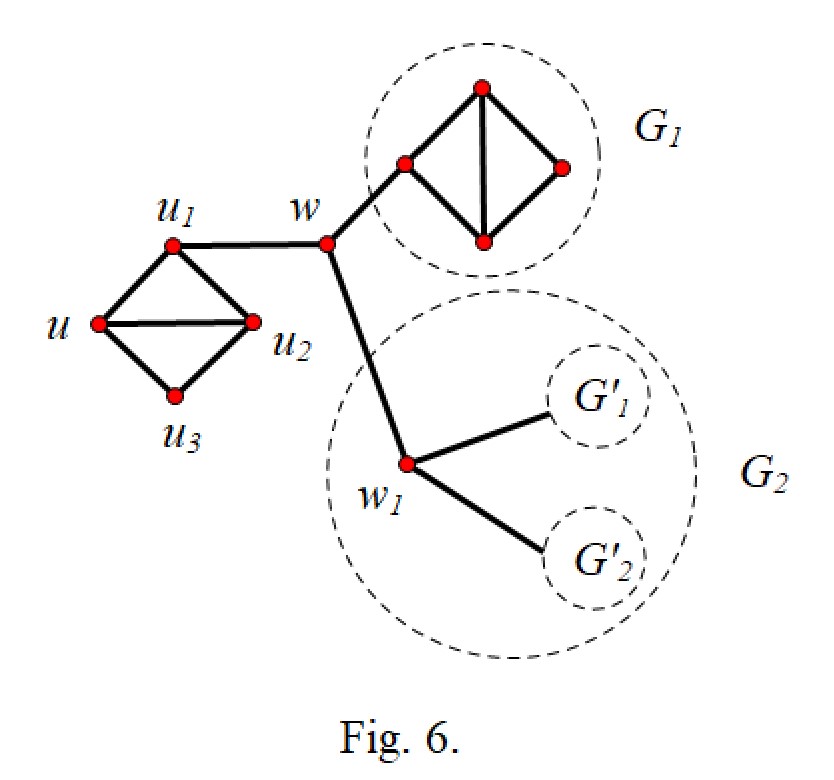}

\end{figure}

If $G''$ is connected and $G''\neq B_2$, by the induction hypothesis, $\iota(G'',B_2)\leq \frac{n-10}{5}$. Then the union of $\{w\}$ and a minimum $B_2$-isolating set of $G''$ is a $B_2$-isolating set of $G$. Thus, $\iota(G,B_2)\leq 1+\frac{n-10}{5}\leq \frac{n}{5}$. Observe that $n=14$ and $\{w,w_1\}$ is a $B_2$-isolating set of $G$ when $G''= B_2$. We also have $\iota(G,B_2)\leq \frac{n}{5}$. Suppose that $G''$ is disconnected. Recall that $\Delta(G)= 3$, then $d(w_1)=3$ and $G''$ has exactly two components. Let $G''=G'_1 + G'_2$ (see Figure 6). Now let us consider the components $G'_1$ and $G'_2$. If $G'_1\neq B_2$ and $G'_2\neq B_2$, then the union of a minimum $B_2$-isolating set of $G'_1$, a minimum $B_2$-isolating set of $G'_2$ and $\{w\}$ is a $B_2$-isolating set of $G$. Thus, by Lemma \ref{lem1}, \ref{lem2} and the induction hypothesis,
$$\iota(G,B_2)\leq 1+\iota(G'_1,B_2)+\iota(G'_2,B_2)\leq 1+\frac{|V(G'_1)|}{5}+\frac{|V(G'_2)|}{5}\leq \frac{n}{5}.$$ If $G'_1= B_2$ and $G'_2= B_2$, we have $n=18$ and $\{w,w_1\}$ is a $B_2$-isolating set of $G$. Hence $\iota(G,B_2)\leq \frac{n}{5}$. So, it remains to consider the case of exactly one of $\{G'_1,G'_2\}$ is isomorphic to $B_2$. Suppose that $G'_1=B_2$ and $G'_2\neq B_2$. Note that the union of a minimum $B_2$-isolating set of $G'_2$, the neighbor of $w_1$ in $G'_1$ and $\{w\}$ is a $B_2$-isolating set of $G$. Therefore, $\iota(G,B_2)\leq 1+1+ \frac{|V(G'_2)|}{5}\leq \frac{n}{5}$. This completes the proof of Case 1.

{\it Case 2.} $d(u_1)=3$ and $d(u_3)=3$

If $N(u_1)=N(u_3)$, denote $G^*=G\setminus (N[u_1]\cup \{u_3\})$. Then $G^*$ is connected and $|G^*|=n-5 \geq 5$ since $\Delta(G)= 3$ and $n\geq 10$. Observe that the union of $\{u_1\}$ and a minimum $B_2$-isolating set of $G^*$ is a $B_2$-isolating set of $G$. Hence, by the induction hypothesis, $\iota(G,B_2)\leq 1+\frac{|V(G^*)|}{5}=\frac{n}{5}$. Otherwise, there exist two vertices $w,z\in V(G)$ such that $u_1$ is adjacent to $w$ and $u_3$ is adjacent to $z$.

We first prove $G-N[u]$ is connected. Let $G'=G-N[u]-w$. Note that this case differs from Case 1 only in that there is an edge between $u_3$ and $G'$. By Lemma \ref{lem5} and the proof of Case 1, we have $\iota(G,B_2)\leq \frac{n}{5}$. Therefore, we omit the proof. Next we treat $G-N[u]$ is disconnected. Since $\Delta(G)=3$, then $G-N[u]$ contains exactly two components and $w$ and $z$ belong to different components. Define $G-N[u]=G_{w}+G_{z}$, where $G_{w}$ contains $w$ and $G_{z}$ contains $z$. Obviously, if $G_{w}=B_2$ or $G_{z}=B_2$, the union of $\{u_1\}$ and a minimum $B_2$-isolating set of $G_{z}$ or the union of $\{u_3\}$ and a minimum $B_2$-isolating set of $G_{w}$ is a $B_2$-isolating set of $G$, respectively. Hence $\iota(G,B_2)\leq \frac{n}{5}$. So, suppose that $G_{w}\neq B_2$ and $G_{z}\neq B_2$. Let $G'=G-V(G_{z})-N[u]-w$. By Lemma \ref{lem5} (3), we have $\iota(G_z,B_2)=\iota(G[V(G_z)\cup \{u,u_2,u_3\}],B_2)$. Similarly, using the same method of Case 1, we have $\iota(G,B_2)\leq \frac{n}{5}$. This completes the proof of Lemma \ref{th3}.
\end{proof}

So far, it remains to consider Theorem \ref{th1} when $\Delta(G)\geq 4$.

\begin{lem}\label{lem4}
The connected graph $Y$ of order 9 has the following properties: \\
(1) $\kappa(Y)=4$,\\
(2) $\Delta(Y)=\delta(Y)=4$,\\
(3) for any two vertices $u,v \in V(Y)$, $|N(u)\cap N(v)|\leq 2$,\\
(4) for any vertex $u\in V(Y)$, there exists a vertex $v\in V(Y)\setminus \{u\}$ such that the graph induced by $V(Y)\setminus (\{u\}\cup N[v])$ is $P_3$.
\end{lem}

\begin{proof}
It is easy to check these properties of the graph $Y$ (see Figure 2).
\end{proof}

\begin{lem}\label{lem6}\cite{CH}
Let $G$ be a graph on $n$ vertices and $\mathcal{F}$ a family of graphs and let $ A\cup B$ be a partition of $V(G)$. Then
$$\iota(G,\mathcal{F})\leq \iota(G[A],\mathcal{F}) + \gamma(G[B]).$$
\end{lem}

\begin{lem}\label{th4}
Let $G$ be a connected graph of order $n$. If $\Delta(G)\geq 4$, then
$$\iota(G,B_2)\leq \frac{n}{5}$$ except for $G=Y$.
\end{lem}

\begin{proof}
Let $G$ be a connected graph of order $n$ with $\Delta(G)\geq 4$. The proof is by induction on $n$. By Lemma \ref{lem3}, the result is trivial if $n \leq 9$ or $\iota(G,B_2)=0$. Thus, suppose that $n\geq 10$ and $\iota(G,B_2)\geq 1$. Denote by $d(u)=\Delta(G)$ and $H=G-N[u]$. Obviously, $\iota(G,B_2)=1$ if $H$ is $B_2$-free. If $H=B_2$ or $K_4$, $\iota(G,B_2)\leq 1+1= 2\leq \frac{n}{5}$ for $n\geq 10$. If $H=Y$, then $\Delta(G)\geq 5$. Hence we have $n\geq 15$ and $\iota(G,B_2)\leq 1+\iota(Y,B_2)=3 \leq \frac{n}{5}$. Suppose that $H\neq B_2, K_4,Y$. By Lemma \ref{th3} and the induction hypothesis, it is easy to check that $\iota(G,B_2)\leq \frac{n}{5}$ when $H$ is connected. Therefore, let $H=G_1+ G_2+ \cdots+ G_k$ with $k\geq 2$ and $|V(G_i)|=n_i$ for $i=1,2,\ldots,k$. If $H$ does not contain $B_2, K_4$ or $Y$ as a component, by Lemma \ref{lem1}, \ref{lem2}, \ref{th3} and the induction hypothesis, we have
$$\iota(G,B_2)\leq |\{u\}|+\sum_{i=1}^{k}\iota(G_i,B_2)\leq 1+\frac{n_1}{5}+\frac{n_2}{5}+\cdots+\frac{n_k}{5}=\frac{n-\Delta(G)+4}{5}.$$
Since $\Delta(G)\geq 4$, then $\iota(G,B_2)\leq \frac{n}{5}$.

Next suppose that at least one component of $H$ is $B_2, K_4$ or $Y$. We sort the components of $H$ in the order of $K_4$, $Y$, $B_2$ with one vertex of degree 3 of $B_2$ is adjacent to one vertex of $N(u)$, $B_2$ with only vertices of degree 2 of $B_2$ are adjacent to vertices of $N(u)$, and others. Then $G_1$ is isomorphic to $K_4$, $Y$, or $B_2$. Let $N(u)=\{u_1,u_2,\ldots,u_{\Delta(G)}\}$. Since $G$ is connected, without loss of generality, suppose $N(u_1)\cap V(G_1)\neq \phi$. Denote by $G^*=G-u_1-V(G_1)$. Obviously, $|V(G^*)|\geq 5$.

{\it Case 1.} $G^*$ is connected.

{\it Subcase 1.1.} $G_1=K_4$. If $G^*= Y$, we have $n=14$ and $\Delta(G)= 5$. By Lemma \ref{lem4} (4), there exists a vertex $v\in V(G^*)$ such that the graph induced by $G^*-u-N[v]$ is $P_3$. Since $\Delta(G)= 5$, we have $\{u_1,v\}$ is a $B_2$-isolating set of $G$ and hence $\iota(G,B_2)\leq 2 \leq \frac{n}{5}$. If $G^*\neq Y$, by the induction hypothesis and Lemma \ref{th3}, $\iota(G^*,B_2)\leq \frac{n-5}{5}$. Then, by Lemma \ref{lem6}, $\iota(G,B_2)\leq \gamma(G[V(G_1)\cup \{u_1\}])+\iota(G^*,B_2)\leq 1+\frac{n-5}{5}=\frac{n}{5}$.

{\it Subcase 1.2.} $G_1=Y$. Let $x$ be a neighbor of $u_1$ in $V(G_1)$. If $G^*= Y$, we have $n=19$ and $\Delta(G)= 5$. Then, by Lemma \ref{lem4} (4), there exist a vertex $v_1\in V(G^*)$ such that the graph induced by $G^*-u-N[v_1]$ is $P_3$ and a vertex $v_2\in V(G_1)$ such that the graph induced by $G_1-x-N[v_2]$ is $P_3$. Then $\{v_1,u_1,v_2\}$ is a $B_2$-isolating set of $G$ and $\iota(G,B_2)\leq 3\leq \frac{n}{5}$. If $G^*\neq Y$, similar to Subcase 1.1, $\iota(G,B_2)\leq \gamma(G[V(G_1)\cup \{u_1\}])+\iota(G^*,B_2)\leq 2+\frac{n-10}{5}=\frac{n}{5}$.

{\it Subcase 1.3.} $G_1=B_2$ and there is one vertex of degree 3 of $V(G_1)$ is adjacent to $u_1$. Let $x$ be a neighbor of $u_1$ in $V(G_1)$ and $d_{G[V(G_1)]}(x)=3$. If $G^*= Y$, we have $n=14$ and $\Delta(G)= 5$. Then, similarly, there exists $v\in V(G^*)$ such that $G^*-u-N[v]=P_3$. Define $P=P_3$. If $G^*-N[u_1]-N[v]$ has no $B_2$, then $\iota(G,B_2)\leq 2\leq \frac{n}{5}$. Otherwise, let $N(x)=\{x_1,x_2,x_3\}$ and let $d_{G[V(G_1)]}(x_2)=3$. Observe that $G^*-N[u_1]-N[v]$ contains $B_2$ if and only if $e(V(P),x_1)=3$ or $e(V(P),x_3)=3$. Assume that $e(V(P),x_3)=3$, then $d(x_3)=5$. By Lemma \ref{lem4} (1), $G-(N[x_3]\setminus\{x\})$ is a connected graph of order 9. Since $G-(N[x_3]\setminus\{x\})\neq Y$, by Lemma \ref{lem1} and \ref{lem3},
$$\iota(G,B_2)\leq |\{x_3\}|+ \iota(G-(N[x_3]\setminus\{x\}),B_2)\leq 1+1\leq \frac{n}{5}.$$ If $G^*\neq Y$, similar to Subcase 1.1, $\iota(G,B_2)\leq \gamma(G[V(G_1)\cup \{u_1\}])+\iota(G^*,B_2)\leq 1+\frac{n-5}{5}=\frac{n}{5}$.

{\it Subcase 1.4.} $G_1=B_2$ and only vertices of degree 2 of $V(G_1)$ are adjacent to $u_1$. Let $x$ be a neighbor of $u_1$ in $V(G_1)$ and $d_{G[V(G_1)]}(x)=2$ and Let $d_{G[V(G_1)]}(x_2)=2$. Note that the two remaining vertices of $V(G_1)\setminus\{x,x_2\}$ have degrees of 3 in $G$. First we prove the case of $u_1\in N(x_2)$ and the case $u_1\notin N(x_2)$ and $|N(x_2)\cap N(u)|\leq 1$ . If $G^*= Y$, we have $n=14$ and $\Delta(G)= 5$. By Lemma \ref{lem4} (4), there exists $v\in V(G^*)$ such that $G^*-u-N[v]=P_3$. Since $\Delta(G)= 5$, then $\{u_1,v\}$ is a $B_2$-isolating set of $G$ and $\iota(G,B_2)\leq 2 \leq \frac{n}{5}$. If $G^*\neq Y$, by the induction hypothesis,  Lemma \ref{lem5} (1) and \ref{th3}, $\iota(G,B_2)\leq 1+\frac{n-5}{5}=\frac{n}{5}$. It remains the case of $u_1\notin N(x_2)$ and $|N(x_2)\cap N(u)|\geq 2$, we will deal with it later.

{\it Case 2.} $G^*$ is disconnected.

It implies that $E(V(G_i),N(u))=E(V(G_i),u_1)$ for some $i\in\{2,3,\ldots,k\}$. Let us denote the components satisfying $E(V(G_i),N(u))=E(V(G_i),u_1)$ as $G_{11},G_{12},\ldots,G_{1t}$, $t\geq 1$. Let $G_u$ be the component contains $u$ in $G^*$. Then $G^*=G_{11}+G_{12}+\cdots+G_{1t}+G_u$. Assume that there are $s_1B_2$, $s_2K_4$ and $s_3Y$ in $\{G_{11},G_{12},\ldots,G_{1t}\}$.

{\it Subcase 2.1.} $G_1=K_4$. Let $x$ be a neighbor of $u_1$ in $V(G_1)$ and let $N(x)=\{x_1,x_2,x_3\}$. It is easy to check that $\iota(G,B_2)\leq1+\frac{|V(G_{11})|}{5}+\cdots+\frac{|V(G_{1t})|}{5}+\frac{|V(G_u)|}{5}=\frac{n}{5}$ if $G_{11},G_{12},\ldots,G_{1t},G_u\notin \{B_2,K_4,Y\}$.
If $G_u=K_4$, then $\Delta(G)=4$. Hence, by Lemma \ref{lem1}, \ref{lem2}, \ref{th3} and the induction hypothesis,
$$\iota(G,B_2)\leq |\{u_1\}|+\sum_{i=1}^{t}\iota(G_{1i}-N[u_1],B_2)\leq 1+s_3+\frac{n-(5+4+4s_1+4s_2+9s_3)}{5}\leq \frac{n}{5}.$$
If $G_u=B_2$, then $\Delta(G)=4$. Note that $G[N[u]\cup V(G_1)]-\{u,u_1,x\}$ contains $B_2$ if and only if $e(u_2,\{x_1,x_2,x_3\})= 2$ or $e(u_4,\{x_1,x_2,x_3\})= 2$. Without loss of generality, suppose $e(u_2,\{x_1,x_2,x_3\})= 2$. Then $d(u_2)=4$ and $G-N[u_2]$ is a connected graph of order $n-5$ or the union of a connected graph of order $n-6$ and an isolated vertex. By Lemma \ref{lem4}, $G-N[u_2]$ does not contain $Y$ as an induced subgraph. Hence, by the induction hypothesis and Lemma \ref{th3},
$$\iota(G,B_2)\leq |\{u_2\}|+\iota(G-N[u_2],B_2)\leq 1+\frac{n-5}{5}=\frac{n}{5}.$$ If $G_u=Y$, by Lemma \ref{lem4} (4), there exists $v\in V(G_u)$ such that $G_u-u-N[v]=P_3$. Note that $\Delta(G)=5$, then
$$\iota(G,B_2)\leq |\{u_1,v\}|+\sum_{i=1}^{t}\iota(G_{1i}-N[u_1],B_2)\leq 2+s_3+\frac{n-(5+9+4s_1+4s_2+9s_3)}{5}\leq \frac{n}{5}.$$
Suppose $G_u\neq B_2,K_4,Y$. Then at least one of $\{s_1,s_2,s_3\}$ is not less than one. Obviously,
$$\iota(G,B_2)\leq |\{u_1\}|+\sum_{i=1}^{t}\iota(G_{1i}-N[u_1],B_2)+ \iota(G_{u},B_2)\leq 1+s_3+\frac{n-(5+4s_1+4s_2+9s_3)}{5}\leq \frac{n}{5}$$ when $e(V(G_u),V(G_1)\setminus \{x\})=0$. For $e(V(G_u),V(G_1)\setminus \{x\})>0$, $G[V(G_u)\cup \{u_1\}\cup V(G_1)]-\{u_1,x\}\neq B_2,K_4$. If $G[V(G_u)\cup \{u_1\}\cup V(G_1)]-\{u_1,x\}=Y$, by Lemma \ref{lem4} (4), there exists $v$ such that $G[V(G_u)\cup \{u_1\}\cup V(G_1)]-\{u_1,x,u\}- N[v]=P_3$. Then we have
$$\iota(G,B_2)\leq |\{u_1,v\}|+\sum_{i=1}^{t}\iota(G_{1i}-N[u_1],B_2)\leq 2+s_3+\frac{n-(2+9+4s_1+4s_2+9s_3)}{5}\leq \frac{n}{5}.$$ Otherwise, $$\iota(G,B_2)\leq |\{u_1\}|+\iota(G[V(G_u)\cup V(G_1)\setminus\{x\}],B_2)+\sum_{i=1}^{t}\iota(G_{1i}-N[u_1],B_2)$$  $$\leq 1+s_3+\frac{n-(2+4s_1+4s_2+9s_3)}{5}\leq \frac{n}{5}.$$

{\it Subcase 2.2.} $G_1=Y$. Note that none of the components of $H$ is $K_4$. It follows that $s_2=0$. Let $x$ be a neighbor of $u_1$ in $V(G_1)$. It is easy to check that $\iota(G,B_2)\leq 2+\frac{|V(G_{11})|}{5}+\cdots+\frac{|V(G_{1t})|}{5}+\frac{|V(G_u)|}{5}=\frac{n}{5}$ if $G_{11},G_{12},\ldots,G_{1t},G_u\notin \{B_2,K_4,Y\}$. Since $\Delta(G)\geq 5$, we have $G_u\neq B_2,K_4$. If $G_u=Y$, then $\Delta(G)=5$. Similar to the proofs of Subcase 1.2 and Subcase 2.1, we have $$\iota(G,B_2)\leq 3+\sum_{i=1}^{t}\iota(G_{1i}-N[u_1],B_2)\leq 3+s_3+\frac{n-(19+4s_1+9s_3)}{5}\leq \frac{n}{5}.$$ Suppose $G_u\neq Y$. Then at least one of $\{s_1,s_3\}$ is not less than one. If $e(V(G_u),V(G_1)\setminus \{x\})=0$, we have $$\iota(G,B_2)\leq \iota(G_1+u_1,B_2)+\iota(G_u,B_2)+\sum_{i=1}^{t}\iota(G_{1i}-N[u_1])\leq 2+s_3+\frac{n-(10+4s_1+9s_3)}{5}\leq \frac{n}{5}.$$ Otherwise, $$\iota(G,B_2)\leq |\{u_1\}|+\iota(G[V(G_u)\cup V(G_1)\setminus\{x\}],B_2)+\sum_{i=1}^{t}\iota(G_{1i}-N[u_1],B_2)$$ $$\leq 1+s_3+\frac{n-(2+4s_1+9s_3)}{5}\leq \frac{n}{5}$$ since the component of $G[V(G_1)\cup\{u_1\}\cup V(G_u)]-\{u_1,x\}$ is not $B_2,K_4$ or $Y$.

{\it Subcase 2.3.} $G_1=B_2$ and there is one vertex of degree 3 of $V(G_1)$ is adjacent to $u_1$. Note that none of the components of $H$ is $K_4$ or $Y$. It follows that $s_2=s_3=0$. Let $x$ be a neighbor of $u_1$ in $V(G_1)$ and $d_{G[V(G_1)]}(x)=3$. It is easy to check that $\iota(G,B_2)\leq\frac{n}{5}$ if $G_{11},G_{12},\ldots,G_{1t},G_u\notin \{B_2,K_4,Y\}$.
If $G_u=K_4$, by the proof of the case of $G_u=B_2$ in Subcase 2.1, we have $\iota(G,B_2)\leq \frac{n}{5}$. If $G_u=B_2$, then $\Delta(G)=4$. Let $N(x)=\{x_1,x_2,x_3\}$ and let $d_{G[V(G_1)]}(x_2)=3$. Define $G'=G[V(G_u)\cup \{u_1\}\cup V(G_1)]-\{u,u_1,x\}$. Note that $G'$ contains $B_2$ when one of the following four cases is true. (1) $u_2$ is adjacent to $x_1$ and $x_2$ and $u_3$ is adjacent to $x_1$. (2) $u_2$ is adjacent to $x_2$ and $x_3$ and $u_3$ is adjacent to $x_3$. (3) $u_3$ is adjacent to $x_1$ and $u_4$ is adjacent to $x_1$ and $x_2$. (4) $u_3$ is adjacent to $x_3$ and $u_4$ is adjacent to $x_2$ and $x_3$. We can see that the proof methods are similar for the four cases. So let us just consider the first case. Then $G-N[u_2]$ is a connected graph with order $n-5$ or the union of a connected graph with order $n-6$ and a isolated vertex. By Lemma \ref{lem4}, $G-N[u_2]$ does not contain $Y$ as an induced subgraph. Thus
$$\iota(G,B_2)\leq|\{u_2\}|+\iota(G-N[u_2],B_2)\leq \frac{n}{5}.$$ If $G_u=Y$, we have $\Delta(G)=5$. By Lemma \ref{lem4} (4), there exists $v\in V(G_u)$ such that $G_u-u-N[v]=P_3$. Denote $P=P_3$. Then $G[V(G_u)\cup\{u_1\}\cup V(G_1)]-\{u,u_1,x\}-N[v]$ contains $B_2$ if and only if $e(x_1,V(P))=3$ or $e(x_3,V(P))=3$. Suppose $e(x_1,V(P))=3$. Then $d(x_1)=5$. By Lemma \ref{lem4} (1), $G-N[x_1]\setminus\{x\}$ is a connected graph with order $n-5$ and $G-N[x_1]\setminus\{x\}\neq Y$. Therefore,
$$\iota(G,B_2)\leq|\{x_1\}|+\iota(G-(N[x_1]\setminus \{x\}),B_2)\leq \frac{n}{5}.$$ Next suppose $G_u\notin \{B_2,K_4,Y\}$. Then $s_1\geq 1$. Obviously, $$\iota(G,B_2)\leq |\{u_1\}|+\iota(G_u,B_2)+\sum_{i=1}^{t}\iota(G_{1i}-N[u_1],B_2)\leq 1+\frac{n-5-4s_1}{5}\leq \frac{n}{5}$$ when $e(V(G_u),V(G_1)\setminus \{x\})=0$. If $e(V(G_u),V(G_1)\setminus \{x\})>0$, then $G[V(G_1)\cup\{u_1\}\cup V(G_u)]-\{u_1,x\}\neq B_2,K_4$. If $G[V(G_1)\cup\{u_1\}\cup V(G_u)]-\{u_1,x\}=Y$, by Lemma \ref{lem4} (4), there exists $v$ such that $G[V(G_1)\cup\{u_1\}\cup V(G_u)]-\{u,u_1,x\}-N[v]=P_3$. Then we have $$\iota(G,B_2)\leq |\{u_1,v\}|+\sum_{i=1}^{t}\iota(G_{1i}-N[u_1],B_2)\leq 2+\frac{n-(2+9+4s_1)}{5}\leq \frac{n}{5}.$$ Otherwise, $\iota(G,B_2)\leq |\{u_1\}|+\iota(G-u_1-x,B_2)\leq 1+\frac{n-(2+4s_1)}{5}\leq \frac{n}{5}$.

{\it Subcase 2.4.} $G_1=B_2$ and only vertex of degree 2 of $V(G_1)$ is adjacent to $u_1$. Let $x$ be a neighbor of $u_1$ in $V(G_1)$ and $d_{G[V(G_1)]}(x)=2$ and Let $d_{G[V(G_1)]}(x_2)=2$. Note that the two remaining vertices of $V(G_1)\setminus\{x,x_2\}$ have degrees of 3 in $G$. First we prove the case of $u_1\in N(x_2)$ and the case $u_1\notin N(x_2)$ and $|N(x_2)\cap N(u)|= 1$. It is easy to check that $\iota(G,B_2)\leq\frac{n}{5}$ if $G_2,\ldots,G_t,G_u\notin \{B_2,K_4,Y\}$. If $G_u=B_2$ or $K_4$, then $\Delta(G)\leq 4$ and hence $$\iota(G,B_2)\leq |\{u_1\}|+\sum_{i=1}^{t}\iota(G_{1i}-N[u_1],B_2)\leq 1+\frac{n-9-4s_1}{5}\leq \frac{n}{5}.$$ If $G_u=Y$, by Lemma \ref{lem5} (3), $\iota(G_u,B_2)=\iota(V(G_u)\cup (V(G_1)\setminus\{x\}),B_2)$. Furthermore, by Lemma \ref{lem4} (4), there exists $v\in V(G_u)$ such that $G_u-u-N[v]=P_3$. Then $$\iota(G,B_2)\leq |\{u_1,v\}|+\sum_{i=1}^{t}\iota(G_{1i}-N[u_1],B_2)\leq 2+\frac{n-5-9-4s_1}{5}\leq \frac{n}{5}.$$ Suppose $G_u\neq B_2,K_4,Y$. Then $s_1\geq 1$. We have $\iota(G,B_2)\leq |\{u_1\}|+\frac{n-5-4s_1}{5}\leq \frac{n}{5}$. It remains the case of $u_1\notin N(x_2)$ and $|N(x_2)\cap N(u)|\geq 2$.

In the end, we deal with the case of $u_1\notin N(x_2)$ and $|N(x_2)\cap N(u)|\geq 2$, whether $G^*$ is connected or not. Assume that $u_2,u_3\in N(x_2)$. Denote by $G''=G-\{u_2,u_3,x_1,x_2,x_3\}$. Obviously, if $G''$ is connected, then $G''\notin \{B_2,K_4,Y\}$ and we have $\iota(G,B_2)\leq 1+ \frac{n-5}{5}=\frac{n}{5}$. If $G''$ is disconnected, it implies that $E(V(G_i),N(u))=E(V(G_i),\{u_2,u_3\})$ for some $i\in\{2,3,\ldots,k\}$. Let us denote the components satisfying $E(V(G_i),N(u))=E(V(G_i),\{u_2,u_3\})$ as $G''_{11},G''_{12},\ldots,G''_{1t}$, $t\geq 1$ and let $G''_u$ be the component contains $u$ in $G''$. Then $G''=G''_{11}+G''_{12}+\cdots+G''_{1t}+G''_u$. Clearly, $G''_u\neq K_4$. By Lemma \ref{lem4} (3), $G''_u\neq Y$. And from the proof of above Subcase 2.4, we have $\iota(G,B_2)\leq \frac{n}{5}$ if any component of $\{G''_{11},G''_{12},\ldots,G''_{1t}\}$ is $ B_2$. In other words, $G''_{1i}\notin \{B_2,K_4,Y\}$ for $i\in \{1,\ldots,t\}$. Now, we distinguish two cases. If $G''_u\neq B_2$, then $$\iota(G,B_2)\leq |\{x_2\}|+\iota(G''_u,B_2)+\sum_{i=1}^{t}\iota(G''_{1i},B_2)\leq 1+\frac{n-5}{5}=\frac{n}{5}.$$ If $G''_u=B_2$, then $\Delta(G)=4$. Note that $|V(G''_u)\cup N[x_2]|=9$ and $\{u_2\}$ is a $B_2$-isolating set of $G[V(G''_u)\cup N[x_2]]$. Hence $$\iota(G,B_2)\leq |\{u_2\}|+\iota(G-(V(G''_u)\cup N[x_2]\setminus\{u_3\}),B_2)\leq 1+\frac{n-8}{5}\leq \frac{n}{5}.$$ This completes the proof of Lemma \ref{th4}.
\end{proof}

{\bf Proof of Theorem \ref{th1}.}
From Lemma \ref{lem9}, we can see that the bound is sharp. Combining the results in Lemma \ref{lem3}, \ref{th3}, \ref{th4}, we obtain Theorem \ref{th1}.
\hfill $\qed$

{\bf Acknowledgement}
This research was supported Science and Technology Commission of Shanghai Municipality (STCSM) grant 18dz2271000.


\begin{thebibliography}{WWW}

\bibitem{BM}J. A. Bondy, U.S.R. Murty, Graph Theory, in: GTM, vol. 244, Springer, 2008.
\bibitem{B}P. Borg, Isolation of cycles, Graphs Combin. 36 (2020), no. 3, 631–637.
\bibitem{BFK}P. Borg, K. Fenech, P. Kaemawichanurat, Isolation of $k$-cliques, Discrete Math. 343 (2020), no. 7, 111879, 5 pp.
\bibitem{BK}P. Borg, P. Kaemawichanurat, Partial domination of maximal outerplanar graphs, Discrete Appl. Math. 283 (2020), 306–314.
\bibitem{CH}Y. Caro, A. Hansberg, Partial domination - the isolation number of a graph, Filomat 31 (2017), no. 12, 3925–3944.
\bibitem{FK}O. Favaron, P. Kaemawichanurat, Inequalities between the $K_k$-isolation number and the independent $K_k$-isolation number of a graph, Discrete Appl. Math. 289 (2021), 93–97.
\bibitem{O}O. Ore, Theory of Graphs, in: American Mathematical Society Colloquium Publications, vol. 38, American Mathematical Society, Providence, RI,1962.
\bibitem{STP}S. Tokunaga, T. Jiarasuksakun, P. Kaemawichanurat, Isolation number of maximal outerplanar graphs, Discrete Appl. Math. 267 (2019), 215–218.
\bibitem{W}D. B. West, Introduction to Graph Theory, Prentice Hall, Inc., 1996.
\bibitem{ZW}G. Zhang, B. Wu, $K_{1,2}$-isolation in graphs, Discrete Appl. Math. 304 (2021), 365–374.
\end{thebibliography}
\end{document}